\documentclass[11pt]{article}      
\usepackage{amsmath,amssymb,amscd,amsthm, graphicx, verbatim, setspace} 
\usepackage{authblk}

\theoremstyle{plain}
\newtheorem{thm}{Theorem}[section]
\newtheorem{lem}[thm]{Lemma}
\newtheorem{cor}[thm]{Corollary}

\newcommand{\Ex}{\operatorname{Ex}}
\newcommand{\F}{\operatorname{F}}

\title{Constructing sparse Davenport-Schinzel sequences}
\date{}
\author{Jesse Geneson\\
\small\tt ISU\\
\small\tt geneson@gmail.com
}
\begin{document}
\maketitle
\begin{abstract}
For any sequence $u$, the extremal function $\Ex(u, j, n)$ is the maximum possible length of a $j$-sparse sequence with $n$ distinct letters that avoids $u$. We prove that if $u$ is an alternating sequence $a b a b \dots$ of length $s$, then $\Ex(u, j, n) = \Theta(s n^{2})$ for all $j \geq 2$ and $s \geq n$, answering a question of Wellman and Pettie [Lower Bounds on Davenport-Schinzel Sequences via Rectangular Zarankiewicz Matrices, Disc. Math. 341 (2018), 1987--1993] and extending the result of Roselle and Stanton that $\Ex(u, 2, n) = \Theta(s n^2)$ for any alternation $u$ of length $s \geq n$  [Some properties of Davenport-Schinzel sequences, Acta Arithmetica 17 (1971), 355--362]. 

Wellman and Pettie also asked how large must $s(n)$ be for there to exist $n$-block $DS(n, s(n))$ sequences of length $\Omega(n^{2-o(1)})$. We answer this question by showing that the maximum possible length of an $n$-block $DS(n, s(n))$ sequence is $\Omega(n^{2-o(1)})$ if and only if $s(n) = \Omega(n^{1-o(1)})$. We also show related results for extremal functions of forbidden 0-1 matrices with any constant number of rows and extremal functions of forbidden sequences with any constant number of distinct letters. 
\end{abstract}

\section{Introduction}
A Davenport-Schinzel sequence of order $s$ is a sequence with no adjacent same letters that avoids alternations of length $s+2$ \cite{DS}. These sequences have a variety of applications and connections to other problems, including upper bounds on the maximum complexity of lower envelopes of sets of polynomials of bounded degree \cite{DS}, the maximum complexity of faces in arrangements of arcs \cite{agsh}, the maximum number of edges in certain $k$-quasiplanar graphs \cite{fox, gpt}, and extremal functions of tuples stabbing interval chains \cite{gen_tuple}. A $DS(n, s)$-sequence is a Davenport-Schinzel sequence of order $s$ with $n$ distinct letters. The function $\lambda_{s}(n)$ is defined as the maximum possible length of a $DS(n, s)$-sequence. Most research on Davenport-Schinzel sequences has focused on when $s$ is fixed. It is known that $\lambda_{1}(n) = n$, $\lambda_{2}(n) = 2n-1$, $\lambda_{3}(n) = 2n \alpha(n) + O(n)$, $\lambda_{4}(n) = \Theta(n 2^{\alpha(n)})$, $\lambda_{5}(n) = \Theta(n \alpha(n) 2^{\alpha(n)})$, and $\lambda_{s}(n) = n 2^{(1 \pm o(1))\alpha^{t}(n) / t!}$ for all $s \geq 6$, where $t = \lfloor \frac{s-2}{2} \rfloor$ \cite{DS,agshsh,niv,pettie}.

A more general upper bound from Davenport and Schinzel \cite{DS, klazar} shows that $\lambda_{s}(n) \leq s \binom{n}{2}+1$, even when $s$ is not fixed. Roselle and Stanton \cite{rs} constructed a family of sequences to prove that if $s \geq n$, then $\lambda_{s}(n) = \Theta(s n^{2})$. For the case of $s = n$, the coefficient of $n^3$ in their lower bound is $1/3$, and it is an open problem \cite{petwel} to determine what is the actual coefficient between $1/3$ and $1/2$. Wellman and Pettie \cite{petwel} proved several bounds for when $s$ is not fixed but sublinear in $n$, including that if $s = \Omega(n^{1/t} (t-1)!)$, then $\lambda_{s}(n)$ is between $\Omega(n^{2} s / (t-1)!)$ and $O(n^{2} s)$. Call a sequence $r$-sparse if every contiguous subsequence of length $r$ has all letters distinct. Let $\lambda_{s}(n, r)$ be the maximum possible length of an $r$-sparse $DS(n, s)$-sequence. Klazar proved for fixed $s, r, t$ that $\lambda_{s}(n, r) = \Theta(\lambda_{s}(n, t))$ for all $t \geq r \geq 2$ \cite{klazar1}, but the proof does not work when $s$ is not fixed. Wellman and Pettie \cite{petwel} asked whether Roselle and Stanton's $\Omega(s n^2)$ bound can be generalized to $r$-sparse $DS(n, s)$-sequences.

In this paper, we answer the question by constructing $r$-sparse $DS(n, s)$-sequences of length $\Omega(s n^{2})$ for $s \geq n$, where the constant in the bound depends on $r$. As a corollary, we obtain that if $s = \Omega(n^{1/t} (t-1)!)$, then there are $r$-sparse $DS(n, s)$-sequences of length $\Omega(n^{2} s / (t-1)!)$ for all $r \geq 2$. Our construction of $r$-sparse $DS(n, s)$-sequences is a specific case of a more general construction that we define for $(r, s)$-formations. An $(r, s)$-formation is a concatenation of $s$ permutations of $r$ distinct letters. The function $\F_{r, s}(n)$ is defined as the maximum possible length of an $r$-sparse sequence with $n$ distinct letters that avoids all $(r, s)$-formations. Similarly we define the function $\F_{r, s, j}(n)$ to be the maximum possible length of a $j$-sparse sequence with $n$ distinct letters that avoids all $(r, s)$-formations. Nivasch \cite{niv} and Pettie \cite{pettie3} found tight bounds on $\F_{r, s}(n)$ for all fixed $r, s > 0$, which are mostly on the same order as the bounds for $\lambda_{s-1}(n)$. Upper bounds on $(r, s)$-formations have been used to find tight bounds on the extremal functions of several families of forbidden sequences \cite{klazar1, gtseq}, including a family of sequences used to bound the maximum number of edges in $k$-quasiplanar graphs in which every pair of edges intersect at most a constant number of times \cite{gpt, fox}.

In Section \ref{longrsf}, we prove that $\F_{r, s, j}(n) = \Theta(s n^r)$ for all $s \geq n^{r-1}$ and $j \geq r$, where the constant in the bound depends on $r$ and $j$. Our answer to the sparsity question from \cite{petwel} follows as a corollary of this result. In Section \ref{bifur}, we answer another question from the same paper \cite{petwel} about a related sequence extremal function. We use the word \emph{block} to refer to a contiguous subsequence of distinct letters. Let $\lambda_{s}(n; m)$ denote the maximum possible length of a $DS(n, s)$ sequence that can be partitioned into $m$ blocks. Wellman and Pettie asked how large must $s$ be for $\lambda_{s}(n;n) = \Omega(n^{2-o(1)})$ \cite{petwel}. We show that $\lambda_{s}(n;n) = \Omega(n^{2-o(1)})$ if and only if $s(n) = \Omega(n^{1-o(1)})$. In addition, we prove related results for extremal functions of forbidden 0-1 matrices with a bounded number of rows and forbidden sequences with a bounded number of distinct letters.

\section{Forbidden $(r, s)$-formations and hypergraph edge coloring}\label{longrsf}

Klazar proved that $\F_{r, s}(n) \leq s n^{r}$ for all $n \geq r$ \cite{klazar, niv}. In the next theorem, we show for $s$ sufficiently large that this bound is tight up to a factor that depends only on $r$ (and not on $s$ or $n$). In order to answer the sparsity question of Wellman and Pettie, we will only need the following results for $r = 2$, but we prove these results for all $r$ since they are of independent interest.

\begin{thm}\label{fbase}
Fix integer $r \geq 2$ and real number $0 < c \leq 1$. Then $\F_{r, s, r}(n) = \Theta(s n^r)$ for $s \geq c n^{r-1}$, where the constant in the lower bound depends on $c$ and $r$.
\end{thm}

\begin{proof}
It suffices to prove that $\F_{r, s, r}(n) = \Omega(s n^{r})$ for $s \geq c n^{r-1}$, where the constant in the bound depends on $c$ and $r$. Let $x, t > 0$ be two parameters that will be chosen at the end of the proof in terms of $n$, $s$, $c$, and $r$. Define $T_{r}(x, t)$ to be the sequence obtained by starting with the empty sequence and then addending $t$ copies of the subsequence $i_{1} i_{2} \dots i_{r}$ for each $1 \leq i_1 < i_2 < \dots < i_{r-1} <  i_{r} \leq x$, for $(i_{1}, i_{2}, \dots, i_{r})$ in lexicographic order. We call the consecutive copies of the subsequence $i_{1} i_{2} \dots i_{r}$ in $T_{r}(x, t)$ a \emph{troop}, and we say that the \emph{position} of $i_{j}$ in this troop is $j$ for each $1 \leq j \leq r$. We call the set of adjacent troops with $i_{1} = a_{1}, i_{2} = a_{2}, \dots, i_{r-1} = a_{r-1}$ the \emph{troop-row} $B_{a_{1}, a_{2}, \dots, a_{r-1}}$. Note that each troop-row contains fewer than $x$ troops, and there are a total of $\binom{x-1}{r-1}$ troop-rows.

First observe that $T_{r}(x, t)$ has length $r t \binom{x}{r}$, and that $T_{r}(x, t)$ is $r$-sparse. The first fact is true by definition. To see that $T_{r}(x, t)$ is $r$-sparse, suppose that some letter $q$ occurs in two adjacent troops $L$ and $R$, with $L$ preceding $R$. If every letter greater than $q$ occurs in $L$, then the position of $q$ in $R$ must be at least the position of $q$ in $L$, since the letters greater than $q$ are the only letters that can occur after $q$ in $R$. Otherwise if some letter $p > q$ does not occur in $L$, then the troop obtained from $L$ by replacing $q$ with $p$ would be after $L$ but before any troop with $q$ in a lesser position. Thus the position of $q$ in $R$ must be at least the position of $q$ in $L$.

Next we explain why the formations on $r$ letters have length less than $2\binom{x-1}{r-1}+t+1$. Let $a_{1} < a_{2} < \dots < a_{r}$ be arbitrary distinct letters in $T_{r}(x, t)$. Note that we can find a longest formation on the letters $a_{1}, a_{2}, \dots, a_{r}$ by searching greedily from left to right in $T_{r}(x, t)$. Suppose that we go through the troop-rows from beginning to end, and we mark troop-rows greedily on the letters wherever the formation length increases by $1$ (in other words, we mark the last letter of each permutation of the formation that we find greedily). Then every troop-row not equal to $B_{a_{1}, a_{2}, \dots, a_{r-1}}$ increases the length of the formation on letters $a_{1}, a_{2}, \dots, a_{r}$ by at most $2$. Troop-row $B_{a_{1}, a_{2}, \dots, a_{r-1}}$ increases the length of the formation on letters $a_{1}, a_{2}, \dots, a_{r}$ by at most $t+1$, so all $r$-tuples of letters in $T_{r}(x, t)$ have formation length less than $2\binom{x-1}{r-1}+t+1$. Since $T_{r}(x, t)$ is an $r$-sparse sequence with $x$ distinct letters and length $r t \binom{x}{r}$ that avoids all $(r, 2\binom{x-1}{r-1}+t+1)$-formations, choosing e.g. $x = \frac{c n}{4}$ and $t = s/2-1$ suffices to give the bound of $\F_{r, s, r}(n) = \Omega(s n^{r})$ for all $s \geq c n^{r-1}$, where the constant in the bound depends on $c$ and $r$.
\end{proof}

The lemma below generalizes the first part of Chang and Lawler's argument for their upper bound of $\lceil{ 3n/2 - 2 \rceil}$ on proper edge-coloring for linear hypergraphs \cite{chlaw}. We use this lemma to make sequences of increasing sparsity in the main theorem for $(r, s)$-formations. In the case $r = 2$, Kahn's theorem \cite{kahn} can be used to obtain a better constant factor for the lower bound in our main theorem.

\begin{lem}\label{chlawext}
Suppose that $H$ is a $k$-uniform hypergraph in which every pair of edges have intersection size at most $y$ for some $1 \leq y < k$. Then it is possible to color the edges of $H$ with $k^{y} n / y!$ colors so that no pair of edges with intersection size $y$ receive the same color. 
\end{lem}

\begin{proof}
We color the edges of $H$ in an arbitrary order. Assume that we next color an edge $e$. Since every pair of edges in $H$ have intersection size at most $y$, there are at most $\lfloor \frac{n-k}{k-y} \rfloor$ edges already assigned colors that meet $e$ at each of the $\binom{k}{y}$ size-$y$ subsets of vertices that are contained in $e$. Thus there will be an unused color for $e$ if $\binom{k}{y} \frac{n-k}{k-y} < k^{y} n / y!$, which holds for all $1 \leq y < k$.
\end{proof}

The inductive construction for the theorem below uses Theorem \ref{fbase} for the base case and Lemma \ref{chlawext} for the inductive step. 

\begin{thm}
Fix integers $q \geq r \geq 2$ and real number $0 < c \leq 1$. Then $\F_{r, s, q}(n) = \Theta(s n^r)$ for $s \geq c n^{r-1}$, where the constant in the lower bound depends on $c$, $r$, and $q$.
\end{thm}

\begin{proof}
The upper bound was already proved in \cite{klazar, niv}, so it suffices to prove that $\F_{r, s, q}(n) = \Omega(s n^r)$ for all $s \geq c n^{r-1}$, where the constant in the lower bound depends on $c$, $r$, and $q$. In Theorem \ref{fbase}, we proved that the theorem is true for $q = r$. For the initial case of our inductive construction ($q = r$), we set $T_{r, r}(x, t) = T_{r}(x, t)$, where $T_{r}(x, t)$ is the same sequence defined in Theorem \ref{fbase}. For every $q \geq r+1$, we will construct $T_{r, q}(x, t)$ so that $T_{r, q}(x, t)$ has length $q t \binom{x}{r}$ and any $r$-tuple of distinct letters $a_1, a_2, \dots, a_r$ in $T_{r, q}(x, t)$ has formation length less than $2\binom{x-1}{r-1}+t+1$, independently of $q$. In addition, $T_{r, q}(x, t)$ will be $q$-sparse but not $(q+1)$-sparse, and $T_{r, q}(x, t)$ will have at most $(q!)^{r-1} x$ distinct letters.

Like $T_{r, r}(x, t)$, the sequences $T_{r, q}(x, t)$ for $q \geq r+1$ also have troops, where each troop consists of a sequence of $q$ distinct letters repeated $t$ times. The number of troops is $\binom{x}{r}$, independently of $q$. In order to construct $T_{r, q}(x, t)$ from $T_{r, q-1}(x, t)$, we treat each troop in $T_{r, q-1}(x, t)$ as an edge in a $(q-1)$-uniform hypergraph. Define $H = (V, E)$ as the $(q-1)$-uniform hypergraph with vertex set equal to the letters of $T_{r, q-1}(x, t)$ and edge set $E$ with $e \in E$ if and only if there is a troop in $T_{r, q-1}(x, t)$ on the letters of $e$. Suppose for inductive hypothesis that $H$ is a hypergraph in which every pair of edges have intersection size at most $r-1$. Then by Lemma \ref{chlawext}, it is possible to color the edges of $H$ with some coloring $f$ using $(q-1)^{r-1} ((q-1)!)^{r-1}x  / (r-1)!$ colors so that no pair of edges with intersection size $r-1$ receive the same color. For each edge $e \in E$, insert the color $f(e)$ after each of the $t$ occurrences in $T_{r, q-1}(x, t)$ of the $q-1$ letters in $e$. The resulting sequence $T_{r, q}(x, t)$ is $q$-sparse but not $(q+1)$-sparse. It has length $q t \binom{x}{r}$ and at most $((q-1)!)^{r-1}x + (q-1)^{r-1} ((q-1)!)^{r-1}x  / (r-1)! \leq (q!)^{r-1} x$ distinct letters. 

For formations, we consider arbirary distinct letters $a_{1} < a_{2} < \dots < a_{r}$ in $T_{r, q}(x, t)$. Note that if all of the letters were also in $T_{r, q-1}(x, t)$, then they have formation length less than $2\binom{x-1}{r-1}+t+1$ by inductive hypothesis. If at least two of the letters are both new to $T_{r, q}(x, t)$, then the maximum possible formation length on $a_{1}, a_{2}, \dots, a_{r}$ is at most $2\binom{x-1}{r-1}$ since each new letter occurs in at most one troop in each troop-row. If there is a single letter $a_{i}$ that was not in $T_{r, q-1}(x, t)$, then there are two cases. If the letters $a_{1}, a_{2}, \dots, a_{r}$ do not all appear in a single troop together, then $a_{1}, a_{2}, \dots, a_{r}$ have formation length at most $2\binom{x-1}{r-1}$ since each new letter occurs in at most one troop in each troop-row. If all of the letters $a_{1}, a_{2}, \dots, a_{r}$ appear in a single troop together, then their maximum formation length is less than $2\binom{x-1}{r-1}+t+1$, since each troop-row not containing that troop contributes at most $2$ to the formation length. Note that the letters $a_{1}, a_{2}, \dots, a_{r}$ cannot all occur together in two troops by the definition of the coloring $f$. 

Let $H'$ be the $q$-uniform hypergraph with vertex set equal to the letters of $T_{r, q}(x, t)$ and edge set $E$ with $e \in E$ if and only if there is a troop in $T_{r, q}(x, t)$ on the letters $e$. $H$ is a hypergraph in which every pair of edges have intersection size at most $r-1$, so $H'$ is also a hypergraph in which every pair of edges have intersection size at most $r-1$ by the definition of the coloring $f$. This completes the induction. The last part of the proof is choosing $x$ and $t$ in terms of $n$ and $s$. Since $T_{r, q}(x, t)$ is a $q$-sparse sequence avoiding formations of length $2\binom{x-1}{r-1}+t+1$ with at most $(q!)^{r-1} x$ distinct letters and length $q t \binom{x}{r}$, choosing e.g. $x = \frac{c n}{4 \times (q!)^{r-1}}$ and $t = s/2-1$ suffices to give the bound of $\F_{r, s, q}(n) = \Omega(s n^r)$ for all $s \geq c n^{r-1}$, where the constant in the lower bound depends on $c$, $r$, and $q$. Note also that even if $T_{r, q}(x, t)$ has fewer than $n$ distinct letters, we can add a sequence of new distinct letters at the end of $T_{r, q}(x, t)$ to increase the number of distinct letters to $n$ without increasing the maximum formation length.
\end{proof}

For any sequence $u$, we define $\Ex(u, j, n)$ to be the maximum possible length of a $j$-sparse sequence with $n$ distinct letters that avoids $u$. The next corollary implies the answer to Wellman and Pettie's sparsity question since $(ab)^s$ is a $(2, s)$-formation.

\begin{cor}
If $u$ is an $(r, s)$-formation and $0 < c \leq 1$, then $\Ex(u, j, n) = \Theta(s n^{r})$ for all $j \geq r$ and $s \geq c n^{r-1}$, where the constants in the bounds depend on $c$, $r$, and $j$.
\end{cor}

\begin{proof}
The upper bound follows since every $(r, r s)$-formation contains $u$, while the lower bound follows from the last theorem since any sequence that avoids all $(r, s)$-formations will also avoid any specific $(r, s)$-formation.
\end{proof}

\begin{thm}\label{mainsp}
Fix integer $j \geq 2$ and real number $0 < c \leq 1$. Then $\lambda_{s}(n, j) = \Theta(s n^2)$ for $s \geq c n$, where the constant in the bound depends on $c$ and $j$.
\end{thm}

\begin{proof}
This follows from the last corollary with $r = 2$.
\end{proof}

One of the constructions in Wellman and Pettie's paper \cite{petwel} is an inductive construction that uses Roselle and Stanton's construction as its initial case. The construction in Theorem \ref{mainsp} can be substituted for Roselle and Stanton's construction in Wellman and Pettie's proof to generalize the result in \cite{petwel}. 

\begin{cor}
If $s = \Omega(n^{1/t} (t-1)!)$, $\lambda_{s}(n, j)$ is between $\Omega(n^{2} s / (t-1)!)$ and $O(n^{2} s)$ for all $j \geq 2$.
\end{cor}

\begin{proof}
Wellman and Pettie proved the case $j = 2$ in Theorem 4.1 of their paper \cite{petwel}. The initial case of their construction $S_{1}(s, q)$ uses the Roselle-Stanton construction $RS(s, q)$ for $q$ a prime power and $s \geq q$. For $j \geq 3$, their construction and the analysis in their proof also work if we replace the Roselle-Stanton construction in their initial case with our construction in Theorem \ref{mainsp} using $c = 1$ for the bound $s \geq c n$. Note that in the part of their construction where the subsequences are concatenated, the order of letters in each subsequence can be chosen to preserve $j$-sparsity.
\end{proof}

\section{The threshold for $\lambda_{s}(n; n) = \Omega(n^{2-o(1)})$}\label{bifur}

In this section, we first show that $\lambda_{s}(n;n) = \Omega(n^{2-o(1)})$ if and only if $s(n) = \Omega(n^{1-o(1)})$. After that, we give some related results for sequences and 0-1 matrices.

If $s = \Omega(n^{1-o(1)})$, then any sequence that has $\min(s, n)$ blocks which contain every letter and $n-\min(s, n)$ empty blocks will avoid alternations of length $s+2$, as long as the letters have reverse order in adjacent blocks. Note that we can delete at most one letter from each block to avoid having any adjacent same letters where blocks meet, so in this case we have $\lambda_{s}(n;n) \geq ns - n = \Omega(n^{2-o(1)})$.

If $s \neq \Omega(n^{1-o(1)})$, then there exists a constant $\alpha < 1$ and an infinite sequence of positive integers $i_1 < i_2 <\dots$ such that $s(i_j) < i_{j}^{\alpha}$ for each $j > 0$. Thus, it suffices to show that for every $0 < \alpha < 1$, there exists a constant $\beta < 2$ such that $\lambda_{n^{\alpha}}(n; n) = O(n^{\beta})$. 

For 0-1 matrices, we say that $A$ \emph{contains} $P$ if $A$ has a submatrix that is equal to $P$ or that can be turned into $P$ by changing some ones to zeroes. Otherwise $A$ \emph{avoids} $P$. Let $ex(n, m, P)$ denote the maximum number of ones in an $n \times m$ 0-1 matrix that avoids $P$, and let $ex(n, P) = ex(n, n, P)$. The Zarankiewicz problem is to find $ex(n, m, R_{a, b})$ for all $a, b$, where $R_{a, b}$ denotes the $a \times b$ matrix of all ones. The best known upper bounds are $ex(n, m, R_{a, b}) \leq (b-1)^{1/a}(n-a+1)m^{1-1/a}+(a-1)m$ \cite{kst}. 

\begin{thm}\label{mainth2}
$\lambda_{n^{\alpha}}(n; n) = O(n^{\frac{3}{2}+\frac{\alpha}{2}})$ for $0 < \alpha < 1$.
\end{thm}

\begin{proof}
In order to prove the theorem, we bound a different extremal function. Define $\lambda'_{s}(n; m)$ to be the maximum possible length of a sequence with $n$ distinct letters that can be partitioned into $m$ blocks and which has no pair of letters that occur in $s+1$ blocks together. The sequence in the definition of $\lambda'_{s}(n; m)$ is allowed to have adjacent same letters where blocks meet. Note that clearly $\lambda_{s}(n; m) \leq\lambda'_{s}(n; m)$, so to prove Theorem \ref{mainth2}, it suffices to prove that $\lambda'_{n^{\alpha}}(n; n) = O(n^{\frac{3}{2}+\frac{\alpha}{2}})$ for $0 < \alpha < 1$. However, $\lambda'_{s}(n; n) = ex(n, R_{2, s+1})$, so the theorem follows from the upper bounds for $ex(n, R_{a, b})$ \cite{kst}.
\end{proof}

Thus we have the answer to the question of Wellman and Pettie.

\begin{cor}
$\lambda_{s}(n;n) = \Omega(n^{2-o(1)})$ if and only if $s(n) = \Omega(n^{1-o(1)})$.
\end{cor}

The same proof implies the next result for 0-1 matrices.

\begin{thm}
$ex(n, R_{2, s}) =  \Omega(n^{2-o(1)})$ if and only if $s(n) = \Omega(n^{1-o(1)})$.
\end{thm}

The next corollary follows since $R_{2,s}$ contains every $2 \times s$ 0-1 matrix.

\begin{cor}
If $P_s$ is any $2 \times s$ 0-1 matrix that has ones in both the first and last columns, then $ex(n, P_{s}) =  \Omega(n^{2-o(1)})$ if and only if $s(n) = \Omega(n^{1-o(1)})$.
\end{cor}

We can also get more general corollaries about 0-1 matrices and sequences using the upper bounds for the Zarankiewicz problem, specifically that $ex(n, R_{t, n^{\alpha}}) = O(n^{2+\frac{\alpha-1}{t}})$.

\begin{cor}
If $P_s$ is any 0-1 matrix with $s$ columns and at most $r$ rows for some constant $r$, and $P_s$ has ones in both the first and last columns, then $ex(n, P_{s}) =  \Omega(n^{2-o(1)})$ if and only if $s(n) = \Omega(n^{1-o(1)})$.
\end{cor}

For any sequence $u$, define $Ex(u, n; m)$ to be the maximum possible length of a sequence on $n$ distinct letters and $m$ blocks that has no subsequence isomorphic to $u$.

\begin{cor}
If $u$ is any sequence of length $s$ with at most $r$ distinct letters for some constant $r$, then $Ex(u, n; n) = \Omega(n^{2-o(1)})$ if and only if $s(n) = \Omega(n^{1-o(1)})$.
\end{cor}

\section*{Acknowledgement}

The author thanks the anonymous referee for many helpful comments and correcting errors in the paper.


\begin{thebibliography}{}
\bibitem{agshsh} P. Agarwal, M. Sharir, and P. Shor. Sharp upper and lower bounds on the length of general Davenport-Schinzel sequences. J. Combin. Theory Ser. A, 52:228-274, 1989.
\bibitem{chlaw} W. Chang and E. Lawler, Edge coloring of hypergraphs and a conjecture of Erd\H{o}s, Faber, Lov\'{a}sz, Combinatorica, 8: 293-295, 1989.
\bibitem{ck} J. Cibulka and J. Kyn\v{c}l. Tight bounds on the maximum size of a set of permutations with bounded {VC}-dimension. Journal of Combinatorial Theory Series A, 119: 1461-1478, 2012.
\bibitem{DS} H. Davenport and A. Schinzel. A combinatorial problem connected with differential equations. American J. Mathematics, 87:684-694, 1965.
\bibitem{erdos} P. Erd\H{o}s, Problems and results in graph theory and combinatorial analysis. Proceedings of the Fifth British Combinatorial Conference, Congress. Numer., 15: 169-192, 1975.
\bibitem{fox} J. Fox, J. Pach, and A. Suk. The number of edges in $k$-quasiplanar graphs. SIAM Journal of Discrete Mathematics, 27:550-561, 2013.
\bibitem{gen_tuple} J. Geneson, A Relationship Between Generalized Davenport-Schinzel Sequences and Interval Chains. Electr. J. Comb. 22(3): P3.19, 2015.
\bibitem{gpt} J. Geneson, R. Prasad, and J. Tidor, Bounding Sequence Extremal Functions with Formations. Electr. J. Comb. 21(3): P3.24, 2014.
\bibitem{gtseq} J. Geneson and P. Tian, Sequences of formation width 4 and alternation length 5. CoRR abs/1502.04095, 2015.
\bibitem{kahn} J. Kahn, Coloring nearly-disjoint hypergraphs with $n + o(n)$ colors, J. Comb. Theory, Ser. A, 59:31-39, 1992.
\bibitem{klazar}  M. Klazar. Generalized Davenport-Schinzel sequences: results, problems, and applications. Integers, 2:A11, 2002.
\bibitem{klazar1} M. Klazar. A general upper bound in the extremal theory of sequences. Commentationes Mathematicae Universitatis Carolinae, 33:737-746, 1992.
\bibitem{kst} T. K\H{o}vari, V. T. S\'{o}s, and P. Tur\'{a}n. On a problem of K. Zarankiewicz. Colloquium Math., 3: 50-57, 1954.
\bibitem{niv}  G. Nivasch. Improved bounds and new techniques for Davenport-Schinzel sequences and their generalizations. J. ACM, 57(3), 2010.
\bibitem{pettie} S. Pettie. Sharp bounds on Davenport-Schinzel sequences of every order. J. ACM, 62(5):36, 2015.
\bibitem{pettie3} S. Pettie. Three generalizations of Davenport-Schinzel sequences. SIAM J. Discrete Mathematics, 29(4):2189-2238, 2015.
\bibitem{rs} D. Roselle and R. Stanton. Some properties of Davenport-Schinzel sequences. Acta Arithmetica, XVII:355-362, 1971.
\bibitem{agsh} M. Sharir and P. Agarwal. Davenport-Schinzel Sequences and their Geometric Applications. Cambridge University Press, 1995.
\bibitem{petwel} J. Wellman and S. Pettie, Lower bounds on Davenport-Schinzel sequences via rectangular Zarankiewicz matrices. Discrete Mathematics 341(7): 1987-1993, 2018.
\end{thebibliography}
\end{document}